\documentclass[12pt]{amsart}
\usepackage{amsmath}
\usepackage{amsfonts}
\usepackage{amssymb}
\usepackage{amsthm}
\usepackage[T1]{fontenc}

\def\N{\mathbb{N}}
\def\F{\mathbb{F}}

\newcommand{\lcm }{\mathrm{lcm\,}}

\newtheorem{theo}{Theorem}[section]
\newtheorem{defi}[theo]{Definition}
\newtheorem{propo}[theo]{Proposition}
\newtheorem{coro}[theo]{Corollary}

\begin{document}
\title[Asymptotically exact sequences]{Families of curves over any finite field
with a class number greater than the  Lachaud - Martin-Deschamps bounds}
\author{St\'ephane Ballet}
\address{Institut de Math\'{e}matiques
de Luminy\\ case 930, F13288 Marseille cedex 9\\ France}
\email{ballet@iml.univ-mrs.fr}
\author{Robert Rolland}
\address{Institut de Math\'{e}matiques
de Luminy\\ case 930, F13288 Marseille cedex 9\\ France}
\email{robert.rolland@acrypta.fr}
\date{\today}

\begin{abstract}
We study and explicitly construct some families of
asymptotically exact sequences of algebraic function fields.
It turns out that these families have an asymptotical class number
widely greater than the general Lachaud - Martin-Deschamps bounds.
We emphasize that we obtain asymptotically exact 
sequences of algebraic function fields over any finite field
$\F_q$, in particular when $q$ is not a square and that these
sequences are dense towers.
\end{abstract}

\maketitle

\section{Introduction}

The algebraic properties of algebraic function fields defined over a finite field is somehow reflected by their numerical properties, 
namely their numerical invariants such as the number of places of degree one over a given ground field extension, 
the number of classes of its Picard group, the number of effective divisors of a given degree and so on. 
When, for a given finite ground field, the sequence of the genus of 
a sequence of algebraic function fields tends to infinity, there exist
 asymptotic formulae for different numerical invariants. 
In \cite{tsfa}, Tsfasman generalizes some 
results on the number of rational points 
on the curves (due to Drinfeld-Vladut \cite{vldr}, 
and Serre \cite{serr2}) and on its Jacobian (due to Vladut \cite{vlad}, Rosemblum and Tsfasman \cite{rots}). 
He gives a formula for the asymptotic number of divisors, 
and some estimates for the number of points in the Poincar\'e filtration. 
In this aim, he introduced the notion of asymptotically exact 
family of curves defined over a finite field.     
Let us recall this notion in the language 
of algebraic function fields.

\begin{defi}
Let ${\mathcal F}/\F_q=(F_k/\F_q)_{k\geq 1}$ be
a sequence of algebraic function fields $F_k/\F_q$ defined 
over $\F_q$ of genus $g_k=g(F_k/\F_q)$. 
We suppose that  the sequence of the genus $g_k$ is an increasing 
sequence growing to infinity. 
The sequence ${\mathcal F}/\F_q$ is said to be asymptotically exact 
if for all $m \geq 1$ the following limit exists:  
$$\beta_m({\mathcal F}/\F_q)= \lim_{k \rightarrow +\infty} \frac{B_m(F_k/\F_q)}{g_k} $$  
where $B_m(F_k/\F_q)$ is the number of places of degree $m$
on $F_k/\F_q$. 

The sequence $(\beta_1,\beta_2,....,\beta_m,...)$ is called
the type  
of the aymptotically exact sequence 
${\mathcal F}/\F_q$.
 \end{defi}

Tsfasman and Vladut in \cite{tsvl} made use of this notion to obtain new
general results on the asymptotic properties of zeta functions 
of curves. 

Note that a simple diagonal argument proves that each sequence of algebraic
function fields of growing genus, defined over a finite field  
admits an asymptotically exact subsequence. Unfortunately,
this extraction method is not really suitable for the two
following reasons.
Firstly, in general we do not obtain by this process
an explicit asymptotically exact sequence 
of algebraic function fields defined over an arbitrary finite field, 
in particular when $q$ is not a square. Secondely, the extracted 
sequence is not a 
sufficiently dense asymptotically exact sequence of 
algebraic function fields defined over an arbitrary finite field, 
namely with a control on the growing of the genus. 
Let us define the notion
of density of a family of algebraic function fields defined over 
a finite field of growing genus:

\begin{defi}
Let ${\mathcal F}/\F_q=(F_k/\F_q)_{k\geq 1}$ be a sequence
 of algebraic function fields $F_k/\F_q$ 
of genus $g_k=g(F_k/\F_q)$, defined over $\F_q$. 
We suppose that  the sequence of genus $g_k$ is an increasing 
sequence growing to infinity. 
Then, the density of the sequence ${\mathcal F}/\F_q$ is 
$$d({\mathcal F}/\F_q)=\liminf_{k \rightarrow +\infty}\frac{g_k}{g_{k+1}}.$$
 \end{defi}

A high density can be a useful property in some applications
of sequences or towers of function fields.  
Until now, no explicit examples of dense asymptotically exact sequences 
${\mathcal F}/\F_q$ have been pointed out unless for the case
$q$ square and type $(\sqrt{q}-1,0,\cdots)$.

\medskip

In section 2, we show that we can construct general families, 
of asymptotically exact sequences of algebraic function 
fields defined over an arbitrary finite 
field $\F_q$ of type $(0,...,\frac{1}{r}(q^{\frac{r}{2}}-1),0,...,0,...)$ 
where $r$ is an integer $\geq 1$. In this aim, we prove
the main theorem \ref{lem00} on sequences of algebraic
function fields such that $\beta_r({\mathcal F}/\F_q)=\frac{1}{r}(q^{}-1)$.
We study for these general families the behaviour of
the class number $h_k$, and we compare our estimation to
the general known bounds of Lachaud - Martin-Deschamps. We also
study the number of effective divisors.

Next, in section 3, we construct explicit examples of very 
dense asymptotically 
exact sequences defined 
over an arbitrary finite field $\F_q$. For this purpose we use towers 
of algebraic function fields having for constant field extension
of a given degree $r$ the densified towers of Garcia-Stichtenoth 
(cf. \cite{gast} and \cite{ball2}). In particular, we construct 
an asymptotically 
exact tower of algebraic function 
fields defined over $\F_2$
with a maximal density . This tower has an interesting application
in the theory of algebraic complexity \cite{baro3}.

\section{General results}

\subsection{New families of asymptotically exact sequences }

First, let us recall certain asymptotic results. 
Let us first give the following result obtained 
by Tsfasman in \cite{tsfa}:
  
 \begin{propo}\label{fond}
 Let ${\mathcal F}/\F_{q}=(F_k/\F_{q})_{k\geq 1}$ be a sequence 
of algebraic function fields of increasing genus $g_k$ growing
to infinity. Let $f$ be a function from $\N$ to $\N$
such that $f(g_k)=o(log(g_k))$.
Then
\begin{equation}
 \limsup_{g_k\rightarrow +\infty} \frac{1}{g_k}
\sum_{m=1}^{f(g_k)} \frac{mB_m(F_k)}{q^{m/2}-1}\leq 1.
\end{equation}
 \end{propo}

Using this proposition we can  obtain the
following main theorem:
 
 \begin{theo}\label{lem00}
 Let $r$ be an integer $\geq 1$ and ${\mathcal F}/\F_q=
(F_k/\F_{q})_{k \geq 1}$ be 
a sequence of algebraic function fields of increasing 
genus defined over $\F_{q}$ such that $\beta_r({\mathcal F}/\F_{q})=\frac{1}{r}(q^{\frac{r}{2}}-1)$. 
 Then $\beta_m({\mathcal F}/\F_{q})=0$ for any integer $m \neq r$. 
In particular, 
the sequence ${\mathcal F}/\F_q$ is asymptotically exact.
\end{theo}

\begin{proof}
Let us fix $m\neq r$ and let us prove that $\beta_m({\mathcal F}/\F_{q})=0$.
We use Proposition \ref{fond} with the constant function $f(g)=s=\max(m,r)$.
Then we get
$$ \limsup_{k\rightarrow +\infty} \frac{1}{g_k}
\sum_{j=1}^{s} \frac{jB_j(F_k)}{q^{j}-1}\leq 1.$$
But by hypothesis
$$ \limsup_{k\rightarrow +\infty} \frac{rB_r(F_k)}{g_k(q^{\frac{r}{2}}-1)}=
\lim_{k\rightarrow +\infty} \frac{rB_r(F_k)}{g_k(q^{\frac{r}{2}}-1)}=1.$$
Then 
$$\limsup_{g_k\rightarrow +\infty} \frac{1}{g_k}
\sum_{1\leq j\leq s; j\neq r} \frac{jB_j(F_k)}{q^{j}-1}=
\lim_{g_k\rightarrow +\infty} \frac{1}{g_k}
\sum_{1\leq j\leq s; j\neq r} \frac{jB_j(F_k)}{q^{j}-1}=0.$$
But
$$\frac{B_m(F_k)}{g_k} \leq \frac{(q^m-1)}{m} 
\left(\frac{1}{g_k}\sum_{1\leq j\leq s; j\neq r}
 \frac{jB_j(F_k)}{q^{j}-1}\right).$$
Hence
$$\lim_{g_k\rightarrow +\infty} \frac{B_m(F_k)}{g_k}=0.$$
\end{proof}
 
Note that for any $k$ the following holds:
$$B_1(F_k/\F_{q^r})=\sum_{i \mid r}iB_i(F_k/\F_q).$$
Then if $\beta_r({\mathcal F}/\F_q)=\frac{1}{r}(q^{\frac{r}{2}}-1)$,
by Theorem \ref{lem00} we conclude that $\beta_1({\mathcal F}/\F_{q^r})$
exists and that
$$\beta_1({\mathcal F}/\F_{q^r})=(q^{\frac{r}{2}}-1).$$
In particular the sequence ${\mathcal F}/\F_{q^r}$
reaches the Drinfeld-Vladut bound
and consequently $q^r$ is a square.

If $\beta_1({\mathcal F}/\F_{q^r})$ exists then 
it does not necessarly imply that $\beta_r({\mathcal F}/\F_{q})$ exists 
but only that  $\lim_{k \rightarrow +\infty} \frac{\sum_{m\vert r} mB_m(F_k/\F_q)}{g_k} $ exists. 
In fact, this converse depends on the defining 
equations of the algebraic function fields  $F_k/\F_q$.

Now, let us give a simple consequence of Theorem \ref{lem00}.

\begin{propo}\label{propotranslat}
 Let $r$ and $i$ be  integers $\geq 1$ such that $i$ divides $r$. 
 Suppose that
${\mathcal F}/\F_q=(F_k/\F_{q})_{k \geq 1}$ is an asymptotically
exact sequence
 of algebraic function fields defined over $\F_{q}$ 
of type $(\beta_1=0,\ldots,\beta_{r-1}=0,
\beta_r=\frac{1}{r}(q^{\frac{r}{2}}-1),\beta_{r+1}=0,\ldots)$. 
 Then the sequence 
${\mathcal F}/\F_{q^i}=(F_k/\F_{q^i})_{k \geq 1}$ 
 of algebraic function field defined over $\F_{q^i}$ 
is asymptotically exact of type
 $(\beta_1=0,..,\beta_{\frac{r}{i}-1}=0,\beta_{\frac{r}{i}}=
\frac{i}{r}(q^{\frac{r}{2}}-1),\beta_{\frac{r}{i}+1}=0,\ldots)$. 
\end{propo}

\begin{proof}
Let us remark that by \cite[Lemma V.1.9, p. 163]{stic}, if $P$
is a place of degree $r'$ of $F/\F_q$, there are $\gcd((r',i))$
places of degree $\frac{r'}{\gcd(r',i)}$ over $P$ in the 
extension $F/\F_{q^i}$. As we are interested by the places of
degree $r/i$ in $F/\F_{q^i}$, let us introduce the set 
$$S=\{r'; r\,\gcd(r',i)=i\,r'\}=\{r'; \lcm(r',i)=r\}.$$
Then, 
$$B_{r/i}(F/\F_{q^i})= 
\sum_{r' \in S}\frac{ir'}{r}B_{r'}(F/\F_q).$$
We know that all the $\beta_j(F/\F_q)=0$ but 
$\beta_r(F/\F_q)=\frac{1}{r}(q^{\frac{r}{2}}-1)$.
Then 
$$\beta_{r/i}(F/\F_{q^i})= i \beta_r(F/\F_q),$$
$$\beta_{r/i}(F/\F_{q^i})= \frac{i}{r}\left(q^{\frac{r}{2}}-1\right).$$

\end{proof}

\subsection{Number of points of the Jacobian}

Now, we are interested by the Jacobian cardinality of the
asymptotically exact sequences ${\mathcal F}/\F_q=(F_k/\F_{q})_{k \geq 1}$ 
of 
type $(0,..,0,\frac{1}{r}(q^{\frac{r}{2}}-1),0,...,0)$.
 
 Let us denote by $h_k=h_k(F_k/\F_q)$ the class number of the 
algebraic function field $F_k/\F_q$.
Let us consider the following quantities 
introduced by Tsfasman in \cite{tsfa}:
 $$H_{inf}=H_{inf}({\mathcal F}/\F_q)=
\liminf _{k \rightarrow +\infty }{\frac{1}{g_k}\log h_k}$$
 $$H_{sup}=H_{sup}({\mathcal F}/\F_q)=
\limsup_{k \rightarrow +\infty }{\frac{1}{g_k}\log h_k}$$
 If they coincides, we just write: 
$$H=H({\mathcal F}/\F_q)=
\lim_{k \rightarrow +\infty}{\frac{1}{g_k}\log h_k}=H_{inf}=H_{sup}.$$

Then under the assumptions of the previous section, 
we obtain the following result on the sequence of class numbers 
of these families of algebraic function fields:

\begin{theo}\label{theoh}
 Let ${\mathcal F}/\F_q=(F_k/\F_{q})_{k \geq 1}$ be a sequence  
 of algebraic function fields of increasing genus defined 
over $\F_{q}$ such that $\beta_r({\mathcal F}/\F_{q})=\frac{1}{r}(q^{\frac{r}{2}}-1)$ where $r$ is an integer. 
 Then, the limit $H$ exists and we have: 
$$H=H({\mathcal F}/\F_q)=
\lim_{k \rightarrow +\infty}{\frac{1}{g_k}\log h_k}=
\log \frac{q^{q^{\frac{r}{2}}}}{(q^r-1)^{\frac{1}{r}(q^{\frac{r}{2}}-1)}}.$$
 \end{theo}

\begin{proof}
By Corollary $1$ in \cite{tsfa}, we know that for any asymptotically 
exact family of algebraic function fields defined over $\F_q$, 
the limit $H$ exists and $H=\lim_{k \rightarrow +\infty}{\frac{1}{g_k}\log h_k}=\log q +\sum_{m=1}^{\infty}\beta_m.\log \frac{q^{m}}{q^{m}-1}$.
Hence, the result follows from Theorem \ref{lem00}. 
\end{proof}

\begin{coro}\label{coroboundLMD}
Let ${\mathcal F}/\F_q=(F_k/\F_{q})_{k \geq 1}$ be a sequence  
 of algebraic function fields of increasing genus defined over $\F_{q}$ such
 that $\beta_r({\mathcal F}/\F_{q})=\frac{1}{r}(q^{\frac{r}{2}}-1)$ 
where $r$ is an integer. 
 Then there exists an integer $k_0$ such that for any integer 
$k\geq k_0$, $$h_k> q^{g_k}$$
 \end{coro}

\begin{proof}
By Theorem \ref{theoh}, we have $\lim_{k \rightarrow +\infty} (h_k)^{\frac{1}{g_k}}= \frac{q^{q^{\frac{r}{2}}}}{(q^r-1)^{\frac{1}{r}(q^{\frac{r}{2}}-1)}}$.
But  $$\frac{q^{q^{\frac{r}{2}}}}{(q^r-1)^{\frac{1}{r}(q^{\frac{r}{2}}-1)}}>
\frac{q^{q^{\frac{r}{2}}}}{(q^r)^{\frac{1}{r}(q^{\frac{r}{2}}-1)}}=q.$$ 
Hence, for a sufficiently large $k_0$, we have for $k \geq k_0$ the
following inequality 
$$(h_k)^{\frac{1}{g_k}}>q.$$
\end{proof}

Let us compare this estimation of $h_k$ to the 
general lower bounds given by
G. Lachaud and M. Martin-Deschamps in \cite{lamd}.

\begin{theo}[Lachaud~ - ~Martin-Deschamps bounds]
 Let $X$ be a projective irreducible and non-singular algebraic curve
defined over the finite field $\F_q$
of genus $g$. Let $J_X$ be the jacobian of $X$
and $h$ the class number $h=|J_X(\F_q)|$. Then
\begin{enumerate}
 \item $h\geq L_1=q^{g-1}\frac{(q-1)^2}{(q+1)(g+1)}$,
 \item $h \geq L_2=\left(\sqrt{q}-1 \right)^2 \,\frac{g^{g-1}-1}{g}\,
\frac{|X(\F_q)|+q-1}{q-1},$,
 \item if $g> \sqrt{q}/2$ and if $B_1(X/\F_q) \geq 1$,
then the following holds:\\ $h \geq L_3=(q^g-1) \frac{q-1}{q+g+gq}$.
\end{enumerate}
 
\end{theo}

Then we can prove that for a family of algebraic
function fields satisfying the conditions of Corollary \ref{coroboundLMD},
the class numbers $h_k$ greatly exceeds the bounds $L_i$.
More precisely

\begin{propo}\label{propoLMDS}
Let ${\mathcal F}/\F_q=(F_k/\F_{q})_{k \geq 1}$ be a sequence  
 of algebraic function fields of increasing genus defined over $\F_{q}$ such
 that $\beta_r({\mathcal F}/\F_{q})=\frac{1}{r}(q^{\frac{r}{2}}-1)$ 
where $r$ is an integer. 
 Then 
\begin{enumerate}
\item for $i=1,3$
$$lim_{k \rightarrow +\infty}\frac{h_k}{L_i}=+\infty,$$
\item for $i=2$ 
the following holds:
\begin{enumerate}
\item if r>1 then  
$$lim_{k \rightarrow +\infty}\frac{h_k}{L_2}=+\infty,$$ 
\item if r=1 then
$$\frac{h_k}{L_2}\geq 2,$$
\end{enumerate}
\end{enumerate}
\end{propo}

\begin{proof}
\begin{enumerate}
\item case $i=1$: 
the following holds
$$L_1=q^{g_k-1}\frac{(q-1)^2}{(q+1)(g_k+1)}=
q^{g_k} \frac{(q-1)^2}{q(q+1)(g_k+1)}
<\frac{q^{g_k}}{(g_k+1)},$$
so, using the previous corollary \ref{coroboundLMD}, we conclude that 
for $k$ large
$$\frac{h_k}{L_1}>g_k$$
and consequently
$$\lim_{k\rightarrow +\infty}\frac{h_k}{L_1}=+\infty;$$
\item case $i=2$:
\begin{enumerate}
\item case $r=1$: in this case, we just bound
the number of rational points by the Weil bound.
More precisely
$$L_2=\left(q+1-2\sqrt{q}\right) \,\frac{q^{g_k-1}-1}{g_k}\,
\frac{B_1(F_k/\F_q)+q-1}{q-1} \leq $$
$$
\left(q+1-2\sqrt{q}\right) \,\frac{q^{g_k-1}-1}{g_k}\,
\frac{2q+2g_k\sqrt{q}}{q-1} <
$$
$$2 \,\frac{q+1-2\sqrt{q}}{(q-1)\sqrt{q}}q^{g_k} \left(1+
 \frac{\sqrt{q}}{g_k}\right );$$
but for all $q \geq 2$ 
$$2 \,\frac{q+1-2\sqrt{q}}{(q-1)\sqrt{q}} < 0.4,$$
then 
$$L_2 < 0.4 \,\left(1+
 \frac{\sqrt{q}}{g_k}\right )q^{g_k},$$
hence
$$\frac{h_k}{L_2} > 2.5\,\frac{g_k}{g_k+\sqrt{q}} $$
which gives the resul;
\item case $r>1$: in this case we know that 
$$\lim_{k\rightarrow +\infty}\frac{B_1(F_k/\F_q)}{g_k}=
\beta_1({\mathcal F}/\F_q)=0.$$
But
$$L_2 < \frac{q+1-2\sqrt{q}}{(q-1)q}q^{g_k}
\frac{B_1(F_k/\F_q)+q-1}{g_k}.$$
Then
$$\frac{h_k}{L_2}> \frac{(q-1)q}{q+1-2\sqrt{q}}\,\,
\frac{g_k}{B_1(F_k/\F_q)+q-1}.$$
We know that
$$\lim_{k \rightarrow +\infty}\frac{g_k}{B_1(F_k/\F_q)+q-1}
=+\infty,$$
then
$$\lim_{k \rightarrow +\infty}\frac{h_k}{L_2}=+\infty.$$
\end{enumerate}
\item case $i=3$:
$$L_3 =(q^g-1) \frac{q-1}{q+g+gq} < \frac{q^{g_k}}{g_k},$$
then for $k$ large
$$\frac{h_k}{L_3}> g_k$$
and consequently
$$\lim_{k\rightarrow +\infty}\frac{h_k}{L_1}=+\infty.$$
\end{enumerate}
\end{proof}

As we see,  if ${\mathcal F}/\F_q=(F_k/\F_{q})_{k \geq 1}$ 
satisfies the assumptions of Theorem \ref{theoh}, we have $h_k>
 q^{g_k}>q^{g_k-1}\frac{(q-1)^2}{(q+1)g_k}$ 
for $k\geq k_0$ sufficiently large. 
In fact, the value $k_0$ depends at least
on the values of $r$ and $q$ and 
we can not know anything about this value in the general case.

\subsection{Number of effective divisors}

We consider now the problem of the determination 
of the zeta-function. It is known that to determinate the 
zeta-function 
of an algebraic function field of genus $g$ defined 
over a finite field, we just have to  study the number 
$A_i$ of effective divisors of degree $i$ for $i=0,...,g-1$.
Let us study the asymptotic situation. 
In this aim, we consider the following asymptotic 
values defined by Tsfasman in \cite{tsfa}, 
where $A_{\mu g_k}$ denotes the number $A_i$ where $i$
the nearest integer from $\mu g_k $:

$$\Delta(\mu)_{inf}({\mathcal F}/\F_q)=\liminf _{k \rightarrow +\infty }\frac{1}{g_k}\log A_{\mu g_k}$$
$$\Delta(\mu)_{sup}({\mathcal F}/\F_q)=\limsup_{k \rightarrow +\infty }{\frac{1}{g_k}}\log A_{\mu g_k}$$
If they coincides, we just write: $$\Delta(\mu)({\mathcal F}/\F_q)=
\lim_{k \rightarrow +\infty}{\frac{1}{g_k}}\log A_{\mu g_k}=\Delta(\mu)_{inf}=\Delta(\mu)_{sup}.$$

Then, we obtain the following result which gives exponential estimates of the number of effective divisors 
in the particular case of asymptically exact sequences defined previously.

\begin{theo}
Let ${\mathcal F}/\F_q=(F_k/\F_{q})_{k \geq 1}$ be a sequence  
 of algebraic function fields of increasing genus defined over $\F_{q}$ 
such that $\beta_r({\mathcal F}/\F_{q})=\frac{1}{r}(q^{\frac{r}{2}}-1)$ 
where $r$ is an integer.
Then, the limit $\Delta(\mu)$ exists.  Moreover, if we set: $$\mu_0=\frac{q^{\frac{r}{2}}-1}{q^r-1}$$ then for $\mu \geq \mu_0$, 
we have:

$$\Delta(\mu)({\mathcal F}/\F_q)=\log \frac{q^{{\mu}+q^{\frac{r}{2}}-1}}{q^r-1}=H-(1-\mu)\log q$$

and for $0 \leq \mu \leq \mu_0$, $$\Delta(\mu)({\mathcal F}/\F_q)=\log \frac{\mu\bigl(\frac{q^{\frac{r}{2}}-1}{\mu} +1\bigl)^{\frac{1}{r}(\mu+q^{\frac{r}{2}}-1)}}{q^{\frac{r}{2}}-1}.$$

More precisely, we have: $$\lim_{k \rightarrow +\infty}A_{\mu g_k}^{\frac{1}{g_k}}=\bigl(\frac{q^{\frac{r}{2}}-1}{\mu} +1\bigl)^{\frac{\mu}{r}}  \bigl(1-\bigl(\frac{q^{\frac{r}{2}}-1}{\mu} +1\bigl)^{-1}\bigl)^{-\frac{(q^{\frac{r}{2}}-1)}{r}}\bigl).$$

\end{theo}

\begin{proof}
It is sufficient to apply Theorem 6 in \cite{tsfa} or Proposition 4.1 in \cite{tsvl} with the tower ${\mathcal F}/\F_q=(F_k/\F_{q})_{k \geq 1}$. 
Then, for the last quantity, we directly apply Theorem 4.1 in \cite{tsvl} with this same tower  ${\mathcal F}/\F_q$.
\end{proof}

\section{Examples of asymptotically exact towers} \label{descenttowergast}

Let us note $\F_{q^2}$ a finite field with $q=p^r$ 
and $r$ an integer.

\subsection{Sequences ${\mathcal F}/\F_q$ with $\beta_2(F/\F_q)=\frac{1}{2}(q-1)$}

We consider the Garcia-Stichtenoth's tower ${T}_{0}$ over $\F_{q^2}$ 
constructed in \cite{gast}. Recall that this tower
is defined recursively in the following way.
We set $F_1=\F_{q^2}(x_1)$ the rational function field over  $\F_{q^2}$,
and for $i \geq 1$ we define 
$$F_{i+1}=F_i(z_{i+1}),$$
where $z_{i+1}$ satisfies the equation
$$z_{i+1}^q+z_{i+1}=x_i^{q+1},$$
with
$$x_i=\frac{z_i}{x_{i-1}} \hbox{ for } i\geq 2.$$

We consider the completed Garcia-Stichtenoth's tower ${T}_1$ 
over $\F_{q^2}$ 
studied in \cite{ball3} 
obtained from ${T}_{0}$
by adjonction of intermediate steps. Namely
we have
$${T}_1:\, F_{1,0}\subset \cdots  
\subset F_{i,0} \subset F_{i,1}\subset\cdots \subset F_{i,s} 
\subset \cdots \subset F_{i,n-1}
\subset F_{i+1,0}\subset \cdots $$
where the steps $F_{i,0}$ are the steps $F_i$ of the Garcia-Stichtenoth's tower
and where $F_{i,s}$ ($1 \leq s \leq n-1$) are the intermediate steps.

Let us denote by $g_k$ the genus of $F_k$ in $T_0$,
by $g_{k,s}$ the genus of $F_{k,s}$ in 
$T_1$ and by $B_1(F_{k,s})$ the number of places of 
degree one of $F_{k,s}$ in $T_1$.

Recall that each extension $F_{k,s}/F_k$ is Galois of degree $p^s$ wih full constant field $\F_{q^2}$. 
Moreover, we know by \cite{balbro} that the descent of the 
definition field of the tower $T_1$ from 
$\F_{q^2}$ to  $\F_{q}$ is possible.
More precisely,  there exists a tower $T_2$ 
defined over $\F_{q}$ given by a sequence:
$${T}_2:\, G_{1,0}\subset \cdots  
\subset G_{i,0} \subset G_{i,1}\subset\cdots \subset G_{i,s} 
\subset \cdots \subset G_{i,n-1}
\subset G_{i+1,0}\subset \cdots $$
defined over the constant fied $\F_q$ and related to
the tower $T_1$ by
$$F_{k,s}=\F_{q^2}G_{k,s} \quad \hbox{for all } k  \hbox{ and } s,$$
namely $F_{k,s}/\F_{q^2}$ is the constant field extension of $G_{k,s}/\F_q$. 
Let us prove a proposition establishing 
that the tower $T_2/\F_q$ is asymptotically exact with good density.

\begin{propo}\label{subfield}
Let $q=p^r$. For any integer $k\geq1$, for any integer $s$ such that $s=0,1,...,r$, 
the algebraic function field $G_{k,s}/\F_{p}$ in the tower $T_2$ has a 
genus $g(G_{k,s})=g_{k,s}$ with $B_{1}(G_{k,s})$ places of degree one, 
$B_{2}(G_{k,s})$ places of degree two such that:
\begin{enumerate}

\item $G_k \subseteq G_{k,s} \subseteq G_{k+1}$ 
with $G_{k,0}=G_k$ and $G_{k+1,0}=G_{k+1}$.


\item $g(G_{k,s})\leq \frac{g(G_{k+1})}{p^{r-s}}+1$ with $g(G_{k+1})=g_{k+1}\leq q^{k+1}+q^{k}$ .  


\item $B_{1}(G_{k,s})+2B_{2}(G_{k,s}) \geq (q^2-1)q^{k-1}p^{s}$.


\item $\beta_2(T_2/\F_q)= \lim_{g_{k,s}\rightarrow +\infty} \frac{B_2(G_{k,s}/\F_q)}{g_k} =\frac{1}{2}(q-1)$.

\item $d(T_2/\F_q)=\lim_{l \rightarrow +\infty}\frac{g(G_{l})}{g(G_{l+1})}=\frac{1}{p}$ where $g(G_l)$ and $g(G_{l+1})$ denote the genus of two consecutive algebraic function fields in $T_2$.
\end{enumerate}  
\end{propo}

\begin{proof}
The property $1)$ follows directly from Theorem 4.3 in  \cite{balbro}. Moreover, by Theorem 2.2 in \cite{ball3}, we have 
$g(F_{k,s})\leq \frac{g(F_{k+1})}{p^{r-s}}+1$ with $g(F_{k+1})=g_{k+1}\leq q^{k+1}+q^{k}$ . Then, 
as the algebraic function field $F_{k,s}$ is a constant field extension of $G_{k,s}$, for any integer $k$ and $s$ 
the algebraic function fields $F_{k,s}$ and $G_{k,s}$ have the same genus. So, the inequality satisfied by the 
genus $g(F_{k,s})$ is also true for the genus $g(G_{k,s})$. Moreover, the number of places of degree one 
$B_1(F_{k,s}/\F_{q^2})$ of $F_{k,s}/\F_{q^2}$ is such that $B_1(F_{k,s}/\F_{q^2})\geq (q^2-1)q^{k-1}p^{s}$. Then, as 
the algebraic function field $F_{k,s}$ is a constant field extension of $G_{k,s}$ of degree $2$, it is clear that for any
integer $k$ and $s$, we have $B_{1}(G_{k,s}/\F_p)+2B_{2}(G_{k,s}/\F_p)\geq (q^2-1)q^{k-1}p^{s}$. 
Moreover, we know that for any integer $k\geq 1$, the number of places of degree one $B_1( G_{k,s}/\F_q)$ of  $G_{k,s}/\F_q$ 
corresponds at most to the number of places of degree one $B_1(F_{k,s}/\F_{q^2})$ of $F_{k,s}/\F_{q^2}$ 
which are totally ramified in the tower $T_1$ by \cite{gast}.  Hence, for any integer $k$ and $s$, we have $B_1( G_{k,s}/\F_q)\leq q+1$ 
and so $\beta_1(T_2/\F_q)=0$. Moreover,  $B_{1}(G_{k,s}/\F_q)+2B_{2}(G_{k,s}/\F_p)= B_1(F_{k,s}/\F_{q^2})$ and 
as by  \cite{gast}, $\beta_1(T_1/\F_{q^2})=A(q^2)$, we have $\beta_2(T_2/\F_p)=\frac{1}{2}(q-1)$.

\end{proof}

\subsection{Sequences ${\mathcal F}/\F_q$ with $\beta_4({\mathcal F}/\F_q)=\frac{1}{4}(q^2-1)$}

\subsubsection{The descent on the definition field $\F_{p}$ of a Garcia-Stichtenoth tower defined over $\F_{q^2}$} \label{sectiondescent3}

Now, we are interested in searching the descent of the 
definition field of the tower $T_1$ 
from $\F_{q^2}$ to  $\F_{p}$ if it is possible. 
In fact, we can not establish a general result 
but we can prove that it is possible in the case 
of caracteristic $2$ which is given by the following result.

\begin{propo}\label{proptour}
Let $p=2$. If $q=p^2$, the descent of the definition 
field of the tower $T_1$ from $\F_{q^2}$ to  $\F_{p}$ 
is possible.
More precisely,  there exists a tower $T_3$ 
defined over $\F_{p}$ given by a sequence:

$$T_3=H_{1,0} \subseteq H_{1,1} 
\subseteq H_{2,0}\subseteq H_{2,1}\subseteq ...$$
defined over the constant fied $\F_p$ and related to
the towers $T_1$ and $T_2$ by
$$F_{k,s}=\F_{q^2}H_{k,s} \quad for~all~k~and~s=0,1,$$
$$G_{k,s}=\F_{q}H_{k,s} \quad for~all~k~and~s=0,1,$$
namely $F_{k,s}/\F_{q^2}$ is the constant field 
extension of $G_{k,s}/\F_q$ and $H_{k,s}/\F_q$ 
and $G_{k,s}/\F_q$ is the constant field extension 
of $H_{k,s}/\F_p$. 
\end{propo}

\begin{proof}
Let $x_1$ be a transcendent element over $\F_2$ and let us set
$$H_1=\F_2(x_1), G_1=\F_4(x_1), F_1=\F_{16}(x_1).$$
We define recursively for $k \geq 1$
\begin{enumerate}
 \item $z_{k+1}$ such that  $z_{k+1}^4+z_{k+1}=x_k^5$,
 \item $t_{k+1}$ such that $t_{k+1}^2+t_{k+1} = x_k^5$\\
(or alternatively $t_{k+1}=z_{k+1}(z_{k+1}+1)$),
 \item $x_{k}=z_{k}/x_{k-1}$ if $k>1$ ($x_1$ is yet defined),
 \item \label{quatre}$H_{k,1}=H_{k,0}(t_{k+1})=H_k(t_{k+1})$,
$H_{k+1,0}= H_{k+1}=H_k(z_{k+1})$,
$G_{k,1}=G_{k,0}(t_{k+1})=G_k(t_{k+1})$,
$G_{k+1,0}= G_{k+1}=G_k(z_{k+1})$,
$F_{k,1}=F_{k,0}(t_{k+1})=F_k(t_{k+1})$,
$F_{k+1,0}= F_{k+1}=F_k(z_{k+1})$.
\end{enumerate}
By \cite{balbro}, the tower $T_1=(F_{k,i})_{k \geq 1, i=0,1}$ 
is the densified Garcia-Stichtenoth's tower
over $\F_{16}$ and the two other towers $T_2$ and $T_3$ 
are respectively the descent of $T_1$
over $\F_4$ and over $\F_2$.  
\end{proof}

\begin{propo}\label{subfield1}
Let $q=p^2=4$. For any integer $k\geq1$, for any 
integer $s$ such that $s=0,1,2$, 
the algebraic function field $H_{k,s}/\F_{p}$ in 
the tower $T_3$ has a 
genus $g(H_{k,s})=g_{k,s}$ with $B_{1}(H_{k,s})$ places 
of degree one, 
$B_{2}(H_{k,s})$ places of degree two and $B_{4}(H_{k,s})$ 
places of degree $4$ such that:
\begin{enumerate}

\item $H_k \subseteq H_{k,s} \subseteq H_{k+1}$ 
with $H_{k,0}=H_k$.


\item $g(H_{k,s})\leq \frac{g(H_{k+1})}{p^{r-s}}+1$ with 
$g(H_{k+1})=g_{k+1}\leq q^{k+1}+q^{k}$ .  


\item $B_{1}(H_{k,s})+2B_{2}(H_{k,s}) +4B_{4}(H_{k,s})\geq 
(q^2-1)q^{k-1}p^{s}$.


\item $\beta_4(T_3/\F_p)= \lim_{g_{k,s}\rightarrow +\infty} \frac{B_4(H_{k,s}/\F_p)}{g_k} =\frac{1}{4}(p^2-1)$.

\item $d(T_3/\F_p)=\lim_{l \rightarrow +\infty}\frac{g(H_{l})}{g(H_{l+1})}=\frac{1}{2}$ 
where $g(H_l)$ and $g(H_{l+1})$ 
denote the genus of two consecutive algebraic 
function fields in $T_3$.
\end{enumerate}  
\end{propo}

\begin{proof}
The property $1)$ follows directly from 
Proposition \ref{proptour}. Moreover, by 
Theorem 2.2 in \cite{ball3}, we have 
$g(F_{k,s})\leq \frac{g(F_{k+1})}{p^{r-s}}+1$ with 
$g(F_{k+1})=g_{k+1}\leq q^{k+1}+q^{k}$ . Then, 
as the algebraic function field $F_{k,s}$ is a 
constant field extension of $H_{k,s}$, for any integer $k$ and $s$ 
the algebraic function fields $F_{k,s}$ and $H_{k,s}$ 
have the same genus. So, the inequality satisfied by the 
genus $g(F_{k,s})$ is also true for the genus $g(H_{k,s})$. 
Moreover, the number of places of degree one 
$B_1(F_{k,s}/\F_{q^2})$ of $F_{k,s}/\F_{q^2}$ is such that $B_1(F_{k,s}/\F_{q^2})\geq (q^2-1)q^{k-1}p^{s}$. Then, as 
the algebraic function field $F_{k,s}$ is a constant 
field extension of $H_{k,s}$ of degree $4$, it is clear that for any
integer $k$ and $s$, we have $B_{1}(H_{k,s}/\F_p)+2B_{2}(H_{k,s}/\F_p)+4B_{4}(H_{k,s}/\F_p)\geq (q^2-1)q^{k-1}p^{s}$. 
Moreover, we know that for any integer $k\geq 1$, 
the number of places of degree one 
$B_1( G_{k,s}/\F_q)$ of  $G_{k,s}/\F_q$ 
corresponds at most to the number of places 
of degree one $B_1(F_{k,s}/\F_{q^2})$ of $F_{k,s}/\F_{q^2}$ 
which are totally ramified in the tower $T_1$ by \cite{gast}.  
Hence, for any integer $k$ and $s$, 
we have $B_1( G_{k,s}/\F_q)\leq q+1$ 
and so $\beta_1(T_2/\F_q)=0$. Then, as the 
algebraic function field $G_{k,s}$ is a constant 
field extension of $H_{k,s}$ of degree $2$,  
it is clear that for any integer $k$ and $s$, we have $B_{1}(H_{k,s}/\F_p)+2B_{2}(H_{k,s}/\F_p)\leq B_1( G_{k,s}/\F_q)$ 
and so
$\beta_1(T_3/\F_p)=\beta_2(T_3/\F_p)=0$. 
Moreover,  $B_{1}(H_{k,s}/\F_p)+2B_{2}(H_{k,s}/\F_p)+4B_{4}(H_{k,s}/\F_p)= B_1(F_{k,s}/\F_{q^2})$ and 
as by  \cite{gast}, $\beta_1(T_1/\F_{q^2})=A(q^2)$, we have $\beta_4(T_3/\F_p)=A(p^4)=p^2-1$.

\end{proof}

\begin{coro}  
Let $T_3/\F_2=(H_{k,s}/\F_{2})_{k\in \N, s=0,1,2}$ be 
the tower defined above. 
Then the tower $T_3/\F_2$ is an asymptotically exact 
sequence of algebraic function fields defined over $\F_2$ 
with a maximal density (for a tower).
\end{coro}

\begin{proof}
It follows from (\ref{quatre}) of Proposition \ref{subfield}.
\end{proof}

\end{document}